\documentclass[12pt]{article}
\usepackage{amssymb,tikz}
\usepackage{amsmath,amsthm}
\usepackage[latin1]{inputenc}
\usepackage{graphicx}
\usepackage{hyperref}
\usepackage{enumerate}
\usepackage{tikz}
\usepackage{mathrsfs}
\usepackage{verbatim}

\hypersetup{colorlinks=true, linkcolor=blue, citecolor=blue, urlcolor=blue}

\newtheorem{remark}{Remark}

\newtheorem{lemma}[remark]{Lemma}
\newtheorem{theorem}[remark]{Theorem}
\newtheorem{proposition}[remark]{Proposition}
\newtheorem{corollary}[remark]{Corollary}

\title{Resolvability and Strong Resolvability in the Direct Product of Graphs}

\author{Dorota Kuziak$^{a}$, Iztok Peterin$^{b,c}$ and Ismael G. Yero$^{d}$\\[0.2cm]
$^{(a)}$ \small{Departament d'Enginyeria Inform\`atica i Matem\`atiques, Universitat Rovira i Virgili}\\
\small{Av. Pa\"isos Catalans 26, 43007 Tarragona, Spain}\\
$^{(b)}$ \small{University of Maribor, FEECS, Smetanova 17, 2000 Maribor, Slovenia}\\
$^{(c)}$ \small{IMFM, Jadranska 19, 1000 Ljubljana, Slovenia}\\
$^{(d)}$ \small{Departamento de Matem\'aticas, EPS, Universidad de C\'adiz}\\
\small{Av. Ram\'on Puyol s/n, 11202 Algeciras, Spain}}

\begin{document}

\maketitle

\begin{abstract}
Given a connected graph $G$, a vertex $w\in V(G)$ distinguishes two different vertices $u,v$ of $G$
if the distances between $w$ and $u$ and between $w$ and $v$ are different. Moreover, $w$ strongly resolves the pair $u,v$
if there exists some shortest $u-w$ path containing $v$ or some shortest $v-w$ path containing $u$.
A set $W$ of vertices is a (strong) metric generator for $G$ if every pair of vertices of $G$ is
(strongly resolved) distinguished by some vertex of $W$. The smallest cardinality of a (strong) metric
generator for $G$ is called the (strong) metric dimension of $G$. In this article we study the (strong)
metric dimension of some families of direct product graphs.
\end{abstract}

\emph{Keywords:} Metric dimension; strong metric dimension; direct product of graphs; strong resolving graph.

\emph{AMS Subject Classification:} 05C12; 05C76.

\section{Introduction and preliminaries}

Given a graph $G$, a vertex $w\in V(G)$ \emph{distinguishes} two different vertices $u,v$ of $G$, if
$d_G(u,w)\ne d_G(v,w)$, where $d_G(x,y)$ represents the number of edges of a shortest $x-y$ path. Now,
a set $S\subset V(G)$ is said to be a \emph{metric generator} for $G$ if any pair of vertices of $G$ is
distinguished by some element of $S$. Metric generators were introduced by Slater in \cite{leaves-trees},
where they were called \emph{locating sets}, and also independently by Harary and Melter in
\cite{harary}, where they were called \emph{resolving sets}. The terminology of metric
generators, which is a more intuitive definition, was first presented in \cite{seb} according to the role
they play inside the graph. This last name arise from the concept of metric generators of metric spaces.
That is, if we consider the distance function $d_G:V\times V\rightarrow \mathbb{N}$, then $(V,d_G)$ is clearly a metric
space. A metric generator with the smallest possible cardinality among all the metric generators for $G$ is
called a \emph{metric basis} of  $G$, and its cardinality the \emph{metric dimension} of $G$, denoted by $dim(G)$.

Another useful terminology regarding the metric generators of graphs is the following one. If
$S=\{w_1, \ldots, w_k\}$ is an ordered set of vertices, then the {\em metric representation} of a vertex
$v\in V(G)$ with respect to $S$ is the vector $(d_G(v,w_1), \ldots, d_G(v,w_k))$. In this sense, a set
$S$ is a metric generator for $G$ if distinct vertices have distinct metric representation with respect to $S$.

It is readily seen that a metric generator for a graph uniquely distinguishes every vertex of the graph. However,
as it was shown in \cite{seb}, metric generators do not necessarily distinguish graphs in the following sense.
That is: ``\emph{for a given  metric generator $T$ of a graph $H$, whenever $H$ is a subgraph of a graph $G$ and
the metric vectors of the vertices of $H$ relative to $T$ agree in both $H$ and $G$, is $H$ an isometric subgraph
of $G$}?  \emph{Even though the metric vectors relative to a metric generator of a graph distinguish all pairs of
vertices in the graph, they  do not uniquely determine all distances in a graph.}''\footnote{A sentence from
\cite{str-dim-cart-dir}.} In connection with this problem, a stronger notion of metric generators was introduced
in \cite{seb}.  A vertex $w\in V(G)$ \emph{strongly resolves} two different vertices $u,v\in V(G)$ if
$d_G(w,u)=d_G(w,v)+d_G(v,u)$ or $d_G(w,v)=d_G(w,u)+d_G(u,v)$, \emph{i.e.}, there exists some shortest $w-u$ path
containing $v$ or some shortest $w-v$ path containing $u$. A set $S$ of vertices in a connected graph $G$ is a
\emph{strong metric generator} for $G$ if every two vertices of $G$ are strongly resolved by some vertex of $S$.
The smallest cardinality of a strong metric generator for $G$ is called \emph{strong metric dimension} and is
denoted by $dim_s(G)$. A \emph{strong metric basis} of $G$ is a strong metric generator for $G$ of cardinality $dim_s(G)$.

Strong metric generators were further studied in \cite{Oellermann}. We now describe the approach developed there,
regarding the transformation of the problem of finding the strong metric dimension of a graph to the vertex
cover problem. A vertex $u$ of $G$ is \emph{maximally distant} from $v$ if for every neighbor $w$ of $u$ it
follows that $d_G(v,w)\le d_G(v,u)$. The collection of all vertices of $G$ that are maximally distant from some vertex of
the graph is called the {\em boundary} of the graph\footnote{In fact, the boundary $\partial(G)$ of a graph was
defined first in \cite{chartrand3} as the subgraph of $G$ induced by the set mentioned in our article with the
same notation. We follow the approach of \cite{bres-klav-tepeh,cphmps} where the boundary of the graph is just the
set of all boundary vertices.}, see \cite{bres-klav-tepeh,cphmps}, and is denoted by $\partial(G)$.
If $u$ is maximally distant from $v$ and $v$ is maximally distant from $u$, then we say that $u$ and $v$ are
\emph{mutually maximally distant}. If $u$ is maximally distant from $v$, and $v$ is not maximally distant from $u$, then
$v$ has a neighbor $v_1$, such that $d_G(v_1, u) > d_G(v,u)$, {\em i.e.}, $d_G(v_1, u) =d_G(v,u)+1$. We can observe that $u$ is
maximally distant from $v_1$. If $v_1$ is not maximally distant from $u$, then $v_1$ has a neighbor $v_2$, such that
$d_G(v_2, u) >d_G(v_1,u)$. A similar procedure allows to construct a sequence of vertices $v_1,v_2, \ldots\,$ such that
$d_G(v_{i+1}, u) > d_G(v_i, u)$ for every $i$. Since the graph $G$ is finite, this sequence ends with some $v_k$. Thus, for all
neighbors $x$ of $v_k$ we have $d_G(v_k,u) \ge d_G(x,u)$, and so $v_k$ is maximally distant from $u$ and $u$ is maximally
distant from $v_k$. As a consequence, every boundary vertex belongs to the set $S=\{u\in V(G):$ there exists $v\in V(G)$
such that $u,v$  are mutually maximally distant$\}$. Certainly every vertex of $S$ is a boundary vertex. For some basic
graph classes, such as complete graphs $K_n$, complete bipartite graphs $K_{r,s}$,  cycles $C_n$ and hypercube graphs
$Q_k$, the boundary is simply the whole vertex set, and the boundary of a tree consists of its leaves.

In this article we use the notion of ``strong resolving graph'' based on a concept introduced in \cite{Oellermann}.
The \emph{strong resolving graph} of $G$, denoted by $G_{SR}$,  has vertex set $V(G_{SR}) = \partial(G)$ where two
vertices $u,v$ are adjacent in $G_{SR}$ if and only if $u$ and $v$ are mutually maximally distant in $G$
\footnote{In fact, according to \cite{Oellermann} the strong resolving graph $G'_{SR}$ of a graph $G$ has vertex set
$V(G'_{SR})=V(G)$ and two vertices $u,v$ are adjacent in $G'_{SR}$ if and only if $u$ and $v$ are mutually maximally
distant in $G$. So, the strong resolving graph defined here is a subgraph of the strong resolving graph defined in
\cite{Oellermann} and can be obtained from the latter graph by deleting its isolated vertices.}.

A set $S$ of vertices of $G$ is a \emph{vertex cover set} of $G$ if every edge of $G$ is incident with at least one
vertex of $S$. The \emph{vertex cover number}  of $G$, denoted by $\beta(G)$, is the smallest cardinality of a vertex
cover set of $G$. We refer to a $\beta(G)$-set in a graph $G$ as a vertex cover set of cardinality $\beta(G)$. Oellermann
and Peters-Fransen \cite{Oellermann} showed that the problem of finding the strong metric dimension of a connected graph
$G$ can be transformed to the problem of finding the vertex cover number of $G_{SR}$.

\begin{theorem}{\em \cite{Oellermann}}\label{lem_oellerman}
For any connected graph $G$,
$dim_s(G) = \beta(G_{SR}).$
\end{theorem}

On the other hand, it was shown in \cite{landmarks} and \cite{Oellermann}, that the problems of computing $dim(G)$
and $dim_s(G)$, respectively, are NP-complete. This suggests finding the (strong) metric dimension for special
classes of graphs or reducing the problem of computing the (strong) metric dimension for a graph to that of other
simpler graphs. That is the case of product graphs, where the study of a given parameter can be, in general,
reduced to study such a parameter for the factors of the product. The (strong) metric dimension of product graphs
has been recently studied in several articles. For instance, the metric dimension of Cartesian, lexicographic,
strong and corona product graphs was studied in \cite{pelayo1}, \cite{JanOmo2012,Saputro2013}, \cite{dim-strong}
and \cite{CMWA}, respectively. Also, the strong metric dimension of Cartesian, strong and corona product graphs
was studied in \cite{Oellermann,str-dim-cart-dir}, \cite{Strong-strong,Kuziak-Erratum} and \cite{StrongCorona+Join}, respectively.
Moreover, in \cite{str-dim-cart-dir} a few results on the strong metric dimension of some particular
cases of direct product graphs were presented. In this paper we continue
with the study of the (strong) metric dimension of direct products of graphs for those suitable cases.

We state some extra terminology and notation which we use throughout the article. Given a simple graph $G=(V,E)$, we
denote two adjacent vertices $u,v$ by $u\sim v$ and, in this case, we say that $uv$ is an edge of $G$, \emph{i.e.}, $uv\in E$.
For a vertex $v\in V,$ the set $N_G(v)=\{u\in V: \; u\sim v\}$ is the \emph{open neighborhood} of $v$ and the set
$N_G[v] = N_G(v)\cup \{v\}$ is the \emph{closed neighborhood} of $v$. The \emph{diameter} of $G$ is defined as
$D(G)=\max_{u,v\in V}\{d_G(u,v)\}$. The vertex  $x\in V$ is \emph{diametrical} in $G$ if there exists $y\in V$ such that
$d_G(x,y)=D(G)$. We say that $G$ is $2$-\emph{antipodal} if for each vertex $x\in V$ there exists exactly one vertex
$y\in V$ such that $d_G(x,y)=D(G)$. The graph complement of a graph $G$ is denoted by $G^c$. The edgeless graph and
the complete graph of order $n$ are denoted by $N_n$ and $K_n$, respectively.

The {\em direct product of two graphs} $G$ and $H$ is the graph $G\times H$, such that
$V(G\times H)=V(G)\times V(H)$ and two vertices $(g,h),(g',h')$ are adjacent in $G\times H$
if and only if $gg'\in E(G)$ and $hh'\in E(H)$. Even basic graph properties, such as connectedness are nontrivial for
the direct product. Indeed, $G\times H$ is not necessarily connected, even if both
factors are. This happens exactly when both factors are bipartite (and
connected) and in this case there are exactly two components (see \cite{Hammack2011}
or \cite{Weic}). Also, the expression
\begin{equation}\label{distance}
\min \{\max \{d_{G}^{e}(g,g'),d_{H}^{e}(h,h')\},
\max \{d_{G}^{o}(g,g'),d_{H}^{o}(h,h')\}\},
\end{equation}%
equal to $d_{G\times H}((g,h),(g',h'))$ (see \cite{Kim}), is much more complicated for the direct product
than for other products. Here $d_{G}^{e}(g,g')$ means the length of a shortest walk of even
length between $g$ and $g'$ in $G$ and $d_{G}^{o}(g,g')$ the length of a shortest odd walk between $g$ and $g'$ in $G$. If such a
walk does not exist, we set $d_{G}^{e}(g,g')$ or $d_{G}^{o}(g,g')$ to be infinite.
In contrast to distances, the direct product is the most natural product for open neighborhoods:
\begin{equation}\label{neigh}
N_{G\times H}(g,h)=N_G(g)\times N_H(h).
\end{equation}

As we observe in the last section, several connections between the direct product of graphs and other
products are possible. In this sense, we recall the definitions of the Cartesian and the lexicographic product.
The \emph{Cartesian product of two graphs} $G$ and $H$ is the graph $G\Box H$, such that $V(G\Box H)=V(G)\times V(H)$
and two vertices $(g,h),(g',h')$ are adjacent in  $G\Box H$ if and only if, either ($g=g'$ and $hh'\in E(H)$) or ($h=h'$ and
$gg'\in E(G)$). The \emph{lexicographic
product of two graphs} $G$ and $H$ is the graph $G\circ H$ with $V(G\circ H)=V(G)\times V(H)$ and two vertices $(g,h),(g',h')$
are adjacent in  $G\circ H$ if and only if, either ($g=g'$ and $hh'\in E(H)$) or $gg'\in E(G)$. For more information on
properties of product graphs we suggest a very interesting book \cite{Hammack2011}.


\section{Metric dimension}

As we have mentioned before, the study of properties regarding distances in the direct product of graphs is quite
complicate and could produce several tedious procedures. Also, we ignore the direct product of two bipartite graphs, which
is not connected, and therefore, not interesting for our purposes. In this section we mainly deal with the case in which at least one
of the factors of the product is a complete graph.

\begin{theorem}\label{K-r-K-n}
Let $r\ge 2$ and $t\ge 3$ be integers with $t\ge r$. If $a$ is the smallest nonnegative integer for which
$r'\leq \left\lfloor\frac{t'}{2}\right\rfloor +1$ where $r'=r-3a$ and $t'=t-3a$, then
\begin{equation} \label{formula}
dim(K_r\times K_t)=\left\{\begin{array}{ll}
                                 t+a-1, & \mbox{if $r'<\left\lfloor\frac{t'}{2}\right\rfloor+1$ or $(r'=\frac{t'-1}{2}+1$ and $t'$ is odd$)$}, \\
                                 & \\
                                 t+a, & \mbox{if $r'=\frac{t'}{2}+1$ and $t'$ is even}.
                               \end{array}
\right.
\end{equation}
\end{theorem}

\begin{proof}
Let $U=\{g_1,\ldots,g_r\}$ and $V=\{h_1,\ldots,h_t\}$ be the vertex sets of $K_r$ and $K_t$, respectively.
Let $a$ be the smallest nonnegative integer for which $r'\leq \left\lfloor\frac{t'}{2}\right\rfloor +1$
where $r'=r-3a$ and $t'=t-3a$. We define the following sets:
\begin{itemize}
\item $S_1=\{(g_{\alpha},h_{2\alpha-1}),(g_{\alpha},h_{2\alpha}),(g_{a+2\alpha-1},h_{2a+\alpha}),(g_{a+2\alpha},h_{2a+\alpha})\,:\,1\le \alpha\le a\}$, 
for the case $a\ne 0$, and if $a=0$, then we assume $S_1=\emptyset$;
\item $S_2=\{(g_{3a+\beta},h_{3a+2\beta-1}),(g_{3a+\beta},h_{3a+2\beta})\,:\, 1\le \beta\le r'-1\}$;
\item $S_3=\{(g_{r-1},h_{\gamma})\,:\, 3a+2r'-1\le \gamma\le t-1\}$.
\end{itemize}

It is clear that the sets $S_1,S_2$ and $S_3$ are pairwise disjoint and that $|S_1|=4a$.
On the other hand, we have $|S_2|=2r'-2$ for the general case, which further gives $|S_2|=2r'-2=t'-1$,
when $r'=\frac{t'-1}{2}+1$ and $t'$ is odd, and $|S_2|=2r'-2=t'$ when $r'=\frac{t'}{2}+1$ and $t'$ is even.
For $|S_3|$ notice that $2r-3a-1=3a+2r'-1\leq \gamma \leq t-1$, which implies that $|S_3|=t-2r+3a+1$.
Also, notice that $S_3$ is nonempty only when $0<t-2r+3a+1=t'-2r'+1$, which gives 
$r'<\left\lfloor\frac{t'}{2}\right\rfloor+\frac{1}{2}<\left\lfloor\frac{t'}{2}\right\rfloor+1$.

For the case $r=t=3k$, let $S=S_1-\{(g_{3k},h_{3k})\}$. In this case we have $a=k, r'=t'=0$ and $|S|=4a-1=3k+a-1=t+a-1$. Let now $r\neq t$ or $r=t\neq 3k$.
For $S=S_1\cup S_2\cup S_3$ we obtain $|S|=4a+2r'-2+t-2r+3a+1=t+a-1$, whenever $S_3\ne \emptyset$ (or equivalently $r'<\left\lfloor\frac{t'}{2}\right\rfloor+1$). For $r'=\frac{t'-1}{2}+1$ and $t'$ is odd we get $|S|=4a+t'-1=t+a-1$, and for $r'=\frac{t'}{2}+1$ and $t'$ is even we get $|S|=4a+t'=t+a$, since $S_3=\emptyset$ in both cases. See Figure \ref{sets-basis} for typical examples. In $(i)$: $a=0$ and the condition $r'<\left\lfloor\frac{t'}{2}\right\rfloor+1$
is fulfilled, in $(ii)$: $a=0$, $t=t'$ is odd and $r'=\frac{t'-1}{2}+1$, in $(iii)$: $a=0$, $t=t'$ is even and $r'=\frac{t'}{2}+1$,
and in $(iv)$: $a=2$  and $r'<\left\lfloor\frac{t'}{2}\right\rfloor+1$.

\begin{figure}[ht!]
\begin{center}
\begin{tikzpicture}[scale=0.6,style=thick]
\def\vr{4pt}
\draw [fill=black] (0,0) circle (\vr);
\draw [fill=white] (1,0) circle (\vr);
\draw [fill=white] (2,0) circle (\vr);
\draw [fill=white] (3,0) circle (\vr);

\draw [fill=black] (0,1) circle (\vr);
\draw [fill=white] (1,1) circle (\vr);
\draw [fill=white] (2,1) circle (\vr);
\draw [fill=white] (3,1) circle (\vr);

\draw [fill=white] (0,2) circle (\vr);
\draw [fill=black] (1,2) circle (\vr);
\draw [fill=white] (2,2) circle (\vr);
\draw [fill=white] (3,2) circle (\vr);

\draw [fill=white] (0,3) circle (\vr);
\draw [fill=black] (1,3) circle (\vr);
\draw [fill=white] (2,3) circle (\vr);
\draw [fill=white] (3,3) circle (\vr);

\draw [fill=white] (0,4) circle (\vr);
\draw [fill=white] (1,4) circle (\vr);
\draw [fill=black] (2,4) circle (\vr);
\draw [fill=white] (3,4) circle (\vr);

\draw [fill=white] (0,5) circle (\vr);
\draw [fill=white] (1,5) circle (\vr);
\draw [fill=black] (2,5) circle (\vr);
\draw [fill=white] (3,5) circle (\vr);

\draw [fill=white] (0,6) circle (\vr);
\draw [fill=white] (1,6) circle (\vr);
\draw [fill=black] (2,6) circle (\vr);
\draw [fill=white] (3,6) circle (\vr);

\draw [fill=white] (0,7) circle (\vr);
\draw [fill=white] (1,7) circle (\vr);
\draw [fill=black] (2,7) circle (\vr);
\draw [fill=white] (3,7) circle (\vr);

\draw [fill=white] (0,8) circle (\vr);
\draw [fill=white] (1,8) circle (\vr);
\draw [fill=white] (2,8) circle (\vr);
\draw [fill=white] (3,8) circle (\vr);


\draw [fill=black] (6,0) circle (\vr);
\draw [fill=white] (7,0) circle (\vr);
\draw [fill=white] (8,0) circle (\vr);
\draw [fill=white] (9,0) circle (\vr);
\draw [fill=white] (10,0) circle (\vr);

\draw [fill=black] (6,1) circle (\vr);
\draw [fill=white] (7,1) circle (\vr);
\draw [fill=white] (8,1) circle (\vr);
\draw [fill=white] (9,1) circle (\vr);
\draw [fill=white] (10,1) circle (\vr);

\draw [fill=white] (6,2) circle (\vr);
\draw [fill=black] (7,2) circle (\vr);
\draw [fill=white] (8,2) circle (\vr);
\draw [fill=white] (9,2) circle (\vr);
\draw [fill=white] (10,2) circle (\vr);

\draw [fill=white] (6,3) circle (\vr);
\draw [fill=black] (7,3) circle (\vr);
\draw [fill=white] (8,3) circle (\vr);
\draw [fill=white] (9,3) circle (\vr);
\draw [fill=white] (10,3) circle (\vr);

\draw [fill=white] (6,4) circle (\vr);
\draw [fill=white] (7,4) circle (\vr);
\draw [fill=black] (8,4) circle (\vr);
\draw [fill=white] (9,4) circle (\vr);
\draw [fill=white] (10,4) circle (\vr);

\draw [fill=white] (6,5) circle (\vr);
\draw [fill=white] (7,5) circle (\vr);
\draw [fill=black] (8,5) circle (\vr);
\draw [fill=white] (9,5) circle (\vr);
\draw [fill=white] (10,5) circle (\vr);

\draw [fill=white] (6,6) circle (\vr);
\draw [fill=white] (7,6) circle (\vr);
\draw [fill=white] (8,6) circle (\vr);
\draw [fill=black] (9,6) circle (\vr);
\draw [fill=white] (10,6) circle (\vr);

\draw [fill=white] (6,7) circle (\vr);
\draw [fill=white] (7,7) circle (\vr);
\draw [fill=white] (8,7) circle (\vr);
\draw [fill=black] (9,7) circle (\vr);
\draw [fill=white] (10,7) circle (\vr);

\draw [fill=white] (6,8) circle (\vr);
\draw [fill=white] (7,8) circle (\vr);
\draw [fill=white] (8,8) circle (\vr);
\draw [fill=white] (9,8) circle (\vr);
\draw [fill=white] (10,8) circle (\vr);


\draw [fill=black] (13,0) circle (\vr);
\draw [fill=white] (14,0) circle (\vr);
\draw [fill=white] (15,0) circle (\vr);
\draw [fill=white] (16,0) circle (\vr);
\draw [fill=white] (17,0) circle (\vr);

\draw [fill=black] (13,1) circle (\vr);
\draw [fill=white] (14,1) circle (\vr);
\draw [fill=white] (15,1) circle (\vr);
\draw [fill=white] (16,1) circle (\vr);
\draw [fill=white] (17,1) circle (\vr);

\draw [fill=white] (13,2) circle (\vr);
\draw [fill=black] (14,2) circle (\vr);
\draw [fill=white] (15,2) circle (\vr);
\draw [fill=white] (16,2) circle (\vr);
\draw [fill=white] (17,2) circle (\vr);

\draw [fill=white] (13,3) circle (\vr);
\draw [fill=black] (14,3) circle (\vr);
\draw [fill=white] (15,3) circle (\vr);
\draw [fill=white] (16,3) circle (\vr);
\draw [fill=white] (17,3) circle (\vr);

\draw [fill=white] (13,4) circle (\vr);
\draw [fill=white] (14,4) circle (\vr);
\draw [fill=black] (15,4) circle (\vr);
\draw [fill=white] (16,4) circle (\vr);
\draw [fill=white] (17,4) circle (\vr);

\draw [fill=white] (13,5) circle (\vr);
\draw [fill=white] (14,5) circle (\vr);
\draw [fill=black] (15,5) circle (\vr);
\draw [fill=white] (16,5) circle (\vr);
\draw [fill=white] (17,5) circle (\vr);

\draw [fill=white] (13,6) circle (\vr);
\draw [fill=white] (14,6) circle (\vr);
\draw [fill=white] (15,6) circle (\vr);
\draw [fill=black] (16,6) circle (\vr);
\draw [fill=white] (17,6) circle (\vr);

\draw [fill=white] (13,7) circle (\vr);
\draw [fill=white] (14,7) circle (\vr);
\draw [fill=white] (15,7) circle (\vr);
\draw [fill=black] (16,7) circle (\vr);
\draw [fill=white] (17,7) circle (\vr);


\draw [fill=black] (20,0) circle (\vr);
\draw [fill=white] (21,0) circle (\vr);
\draw [fill=white] (22,0) circle (\vr);
\draw [fill=white] (23,0) circle (\vr);
\draw [fill=white] (24,0) circle (\vr);
\draw [fill=white] (25,0) circle (\vr);
\draw [fill=white] (26,0) circle (\vr);
\draw [fill=white] (27,0) circle (\vr);

\draw [fill=black] (20,1) circle (\vr);
\draw [fill=white] (21,1) circle (\vr);
\draw [fill=white] (22,1) circle (\vr);
\draw [fill=white] (23,1) circle (\vr);
\draw [fill=white] (24,1) circle (\vr);
\draw [fill=white] (25,1) circle (\vr);
\draw [fill=white] (26,1) circle (\vr);
\draw [fill=white] (27,1) circle (\vr);

\draw [fill=white] (20,2) circle (\vr);
\draw [fill=black] (21,2) circle (\vr);
\draw [fill=white] (22,2) circle (\vr);
\draw [fill=white] (23,2) circle (\vr);
\draw [fill=white] (24,2) circle (\vr);
\draw [fill=white] (25,2) circle (\vr);
\draw [fill=white] (26,2) circle (\vr);
\draw [fill=white] (27,2) circle (\vr);

\draw [fill=white] (20,3) circle (\vr);
\draw [fill=black] (21,3) circle (\vr);
\draw [fill=white] (22,3) circle (\vr);
\draw [fill=white] (23,3) circle (\vr);
\draw [fill=white] (24,3) circle (\vr);
\draw [fill=white] (25,3) circle (\vr);
\draw [fill=white] (26,3) circle (\vr);
\draw [fill=white] (27,3) circle (\vr);

\draw [fill=white] (20,4) circle (\vr);
\draw [fill=white] (21,4) circle (\vr);
\draw [fill=black] (22,4) circle (\vr);
\draw [fill=black] (23,4) circle (\vr);
\draw [fill=white] (24,4) circle (\vr);
\draw [fill=white] (25,4) circle (\vr);
\draw [fill=white] (26,4) circle (\vr);
\draw [fill=white] (27,4) circle (\vr);

\draw [fill=white] (20,5) circle (\vr);
\draw [fill=white] (21,5) circle (\vr);
\draw [fill=white] (22,5) circle (\vr);
\draw [fill=white] (23,5) circle (\vr);
\draw [fill=black] (24,5) circle (\vr);
\draw [fill=black] (25,5) circle (\vr);
\draw [fill=white] (26,5) circle (\vr);
\draw [fill=white] (27,5) circle (\vr);

\draw [fill=white] (20,6) circle (\vr);
\draw [fill=white] (21,6) circle (\vr);
\draw [fill=white] (22,6) circle (\vr);
\draw [fill=white] (23,6) circle (\vr);
\draw [fill=white] (24,6) circle (\vr);
\draw [fill=white] (25,6) circle (\vr);
\draw [fill=black] (26,6) circle (\vr);
\draw [fill=white] (27,6) circle (\vr);

\draw [fill=white] (20,7) circle (\vr);
\draw [fill=white] (21,7) circle (\vr);
\draw [fill=white] (22,7) circle (\vr);
\draw [fill=white] (23,7) circle (\vr);
\draw [fill=white] (24,7) circle (\vr);
\draw [fill=white] (25,7) circle (\vr);
\draw [fill=black] (26,7) circle (\vr);
\draw [fill=white] (27,7) circle (\vr);

\draw [fill=white] (20,8) circle (\vr);
\draw [fill=white] (21,8) circle (\vr);
\draw [fill=white] (22,8) circle (\vr);
\draw [fill=white] (23,8) circle (\vr);
\draw [fill=white] (24,8) circle (\vr);
\draw [fill=white] (25,8) circle (\vr);
\draw [fill=black] (26,8) circle (\vr);
\draw [fill=white] (27,8) circle (\vr);

\draw [fill=white] (20,9) circle (\vr);
\draw [fill=white] (21,9) circle (\vr);
\draw [fill=white] (22,9) circle (\vr);
\draw [fill=white] (23,9) circle (\vr);
\draw [fill=white] (24,9) circle (\vr);
\draw [fill=white] (25,9) circle (\vr);
\draw [fill=white] (26,9) circle (\vr);
\draw [fill=white] (27,9) circle (\vr);


\draw (1.5,-1.5) node {$(i)$};
\draw (8,-1.5) node {$(ii)$};
\draw (15,-1.5) node {$(iii)$};
\draw (23,-1.5) node {$(iv)$};
\end{tikzpicture}
\end{center}
\caption{The set $S$ is drawn in bold (edges have been omitted).} \label{sets-basis}
\end{figure}
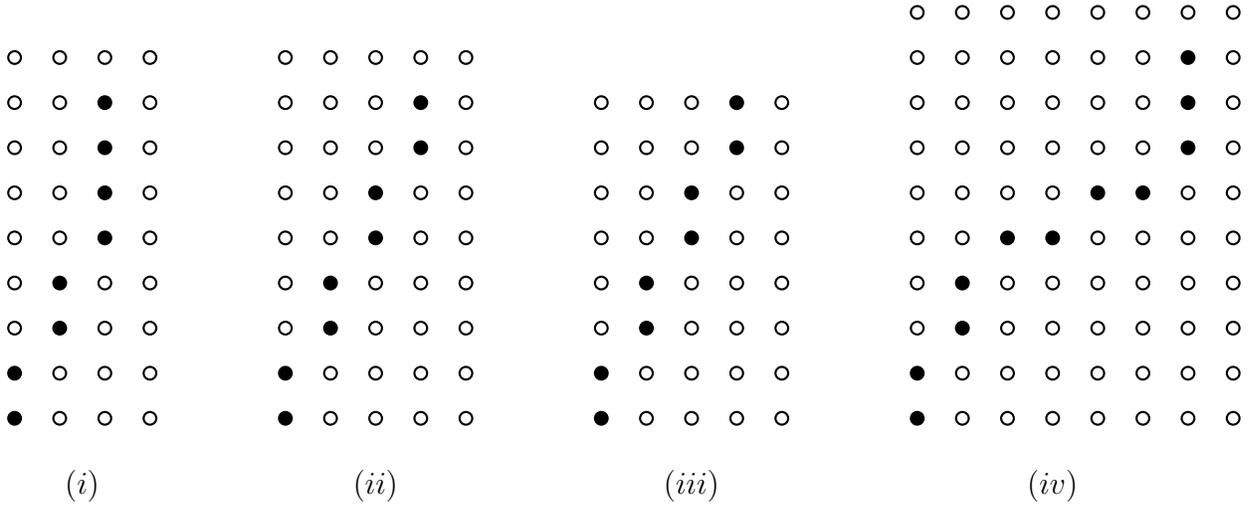

Next we show that $S$ is a metric generator for $K_r\times K_t$, which gives the upper bound for $dim(K_r\times K_t)$ in formula (\ref{formula}).
According to the fact that $d((g_i,h_j),(g_k,h_{\ell}))=1$ if $i\ne k$ and $j\ne \ell$, and that
$d((g_i,h_j),(g_k,h_{\ell}))=2$ if either $i=k$ or $j=\ell$, it is easy to check that the set $S$ given as described above is a
metric generator for $K_r\times K_t$. For instance, if
$(g_i,h_j),(g_k,h_{\ell})\in \{(g_x,h_y)\,:\,\left\lceil\frac{y}{2}\right\rceil<x \mbox{ or } x=r\}$ are different vertices
(when $a=0$ and $r<\left\lfloor\frac{t}{2}\right\rfloor+1$), then at least one of the vertices of the set
$$\{(g_i,h_{2i-1}), (g_i,h_{2i}), (g_k,h_{2k-1}), (g_k,h_{2k}), (g_{\left\lceil j/2\right\rceil},h_j), (g_{\left\lceil \ell/2\right\rceil},h_{\ell})\}\cup \{(g_{r-1},h_{\beta})\,:\,2r-1\le \beta\le t-1\}$$
belongs to $S$ and distinguishes the pair $(g_i,h_j),(g_k,h_{\ell})$. Also, when $a>0$, then
$S_1\neq \emptyset$ and every vertex $(g_x,h_y)$, $x\leq 3a$ or $y\leq 3a$ is distinguished from any other 
vertex of $V(G)\times V(H)$ by using a vertex of $S_1$. The remaining possibilities are clearly observed and are left to the reader.

We now show the lower bound for $dim(K_r\times K_t)$ of formula (\ref{formula}) by induction on $a$.
Let first $a=0$, which means that $r\le \left\lfloor\frac{t}{2}\right\rfloor+1$. Notice that for any two vertices 
$(g_i,h_j),(g_i,h_{\ell})$, with $j\ne \ell$,
it follows that $d((g_i,h_j),(g_k,h_q))=d((g_i,h_{\ell}),(g_k,h_q))$ whenever $q\notin \{j,\ell\}$. Thus, any
metric basis for $K_r\times K_t$ must have cardinality at least $t-1$ and so, $dim(K_r\times K_t)\ge t-1$.
Hence, if $r<\left\lfloor\frac{t}{2}\right\rfloor+1$ or $(r=\frac{t-1}{2}+1$ and $t$ is odd$)$, then $dim(K_r\times K_t)=t-1$.
Notice that the same argumentation with commutativity gives also that $dim(K_r\times K_t)\ge r-1$

We consider the case $t$ even and $r=\frac{t}{2}+1$. Suppose $dim(K_r\times K_t)=t-1$ and let $S$ be a metric 
basis of $K_r\times K_t$. Hence, there exists a vertex of $K_t$, say $h_1$, which is not a projection to $V(K_t)$ 
of any vertex of $S$. By the same argument, all \textcolor[rgb]{1.00,0.00,0.00}{vertices of $K_t$ but $h_1$} are
projections of exactly one vertex of $S$. On the other hand, if there exists two different vertices in $K_r$, 
say $g_1$ and $g_2$, which are projection of exactly one vertex of $S$, say $(g_1,h_i)$ and $(g_2,h_j)$, 
respectively, with $i\neq 1$, $j\neq 1$ and $i\neq j$, then $(g_2,h_i)$ and $(g_1,h_j)$ are not distinguish by 
$S$, which is not possible. So, there exists at most one vertex in $K_r$, say $g_1$, which is a projection 
of exactly one vertex, say $(g_1,h_j)$, of $S$. If there is one vertex in $K_r$, say $g_2$, which is not a 
projection of any vertex of $S$, then $S$ does not distinguish the vertices $(g_2,h_j)$ and $(g_1,h_1)$, 
which is a contradiction. As a consequence, every vertex of $K_r$ is a projection of a vertex in $S$ and 
all but possibly one vertex of $K_r$ are projections of at least two vertices of $S$, which means that 
$|S|\ge 2(r-1)+1=2r-1$. Since $r=\frac{t}{2}+1$, we obtain that $|S|\ge t+1$, a contradiction. Thus 
$dim(K_r\times K_t)=|S|\ge t$ and we are done for $a=0$.


Let now $a>0$, which means that $r>\left\lfloor\frac{t}{2}\right\rfloor+1$. By the induction hypothesis the formula (\ref{formula}) holds for $dim(K_{r-3}\times K_{t-3})$ and we have
\begin{equation*} 
dim(K_{r-3}\times K_{t-3})\geq \left\{\begin{array}{ll}
                                 (t-3)+(a-1)-1, & \mbox{if $r'<\left\lfloor\frac{t'}{2}\right\rfloor+1$ or $(r'=\frac{t'-1}{2}+1$ and $t'$ is odd$)$}, \\
                                 & \\
                                 (t-3)+(a-1), & \mbox{if $r'=\frac{t'}{2}+1$ and $t'$ is even}.
                               \end{array}
\right.
\end{equation*}
It is easy to see that these lower bounds above differ by exactly four to values of formula (\ref{formula}), since $(t-3)'=t-3-3(a-1)=t-3a=t'$ and similarly $(r-3)'=r'$.

Let $S$ be a metric basis of $K_r\times K_t$. Since $a>0$, we have $t\geq r\geq 3$ and $r>\left\lfloor\frac{t}{2}\right\rfloor+1$.
It is easy to see that $dim(K_3\times K_3)\geq 3$, by the same arguments as for $a=0$. Notice that for $t>r=3$ we have $a=0$ and
we can assume that $t\geq r\geq 4$. Similarly, $dim(K_4\times K_4)\geq 4$ and for $t>r=4$ we have again $a=0$ and
we can assume that $t\geq r\geq 5$. Now, suppose that \textcolor[rgb]{1.00,0.00,0.00}{to every vertex of $K_t$ project at most one vertex of $S$.} Since
$r>\left\lfloor\frac{t}{2}\right\rfloor+1$, there exist at least two vertices of $K_r$ to which no vertex of $S$ projects or
there are at least two vertices of $K_r$ to which exactly one vertex of $S$ projects. Both options are not possible for a metric basis
$S$ of $K_r\times K_t$ by using the same arguments as for $a=0$. Thus, there exists a vertex of $K_t$, say $h_i$, to which at least two vertices of
$S$ project. Moreover, there exists a vertex of $K_r$, say $g_1$, to which at least two vertices of $S$ project, otherwise we have a previous
case by $t\geq r$ or more than one vertex in $K_t$ to which no vertex of $S$ projects. We can choose notation so that
$(g_1,h_1)(g_1,h_2)\in S$.

If $i\notin \{1,2\}$, then we can assume that $i=3$ and that $(g_2,h_3)(g_3,h_3)\in S$. Notice that these four vertices are at distance one to every
vertex in $\{g_4,\ldots,g_r\}\times \{h_4,\ldots,h_t\}$, which induces a graph isomorphic to $K_{r-3}\times K_{t-3}$. Thus, we are done by the induction hypothesis.

Finally let $i\in \{1,2\}$, say $i=1$. We may assume that $(g_2,h_1)\in S$. Vertices of 
$\{g_3,\ldots,g_r\}\times \{h_3,\ldots,h_t\}$ induce a graph isomorphic to $K_{r-2}\times K_{t-2}$ 
and all such vertices are at distance one to $(g_1,h_1),(g_1,h_2),(g_2,h_1)$ (and also to $(g_2,h_2)$). 
Thus, $S\cap (\{g_3,\ldots,g_r\}\times \{h_3,\ldots,h_t\})$ is a metric generator for $K_{r-2}\times K_{t-2}$, 
or we can find a metric generator $S'$ for $K_{r-2}\times K_{t-2}$ of cardinality 
$|S-(S\cap \{(g_1,h_1),(g_1,h_2),(g_2,h_1),(g_2,h_2)\})|$, just by including in $S'$ a vertex 
$(g_k,h_l)\notin S$, with $k,l\notin \{1,2\}$, for each vertex $(g_1,h_l)$, $(g_2,h_l)$, $(g_k,h_1)$ 
or $(g_k,h_2)$ which belongs to $S$. Now, let $a'$ be the smallest nonnegative integer for which 
$(r-2)'\leq \left\lfloor\frac{(t-2)'}{2}\right\rfloor +1$ where $(r-2)'=r-2-3a'$ and $t'=t-2-3a'$. 
Recall that $r>\left\lfloor\frac{t}{2}\right\rfloor+1$.
Let first $r=\left\lfloor\frac{t}{2}\right\rfloor+2$. By an easy computation it follows that
$r-3=\left\lfloor\frac{t}{2}\right\rfloor-3+\textcolor[rgb]{1.00,0.00,0.00}{2}\leq \left\lfloor\frac{t-3}{2}\right\rfloor+1$ holds, 
which implies that $a=1$.  Similarly,
$r-2=\left\lfloor\frac{t-2}{2}\right\rfloor+1$ holds, which gives $a'=0$. Thus, we have that 
$|S|\ge 3+|S\cap \{g_3,\ldots,g_r\}\times \{h_3,\ldots,h_t\}|\ge 3+dim(K_{r-2}\times K_{t-2})$ 
or $|S|\ge 3+|S'|\ge 3+dim(K_{r-2}\times K_{t-2})$ and, by induction hypothesis we obtain that 
$dim(K_{r}\times K_{t})=|S|\geq 3+t-2+a'=t+1=t+a$ \textcolor[rgb]{1.00,0.00,0.00}{when $t$ is even and 
$dim(K_{r}\times K_{t})=|S|\geq 3+t-2+a'-1=t=t+a-1$ when $t$ is odd}.

Let now $r>\left\lfloor\frac{t}{2}\right\rfloor+2$. Again it is easy to see that
$r-2>\left\lfloor\frac{t-2}{2}\right\rfloor+1$ and therefore $a'>0$. Since $t-2\geq r-2\geq 3$, there again exist at least one vertex in
$\{g_3,\ldots,g_r\}$ or in $\{h_3,\ldots,h_t\}$, such that at least two vertices of $S$ project to him. By the commutativity we may assume
that this vertex is $g_3$ and that $(g_3,h_j),(g_3,h_k)\in S$. If $j$ and $k$ differ from one, then we can repeat the arguments for
vertices $(V(K_r)-\{g_1,g_2,g_3\})\times (V(K_t)-\{h_1,h_j,h_k\})$ which induce $K_{r-3}\times K_{t-3}$ and we are done by induction.
Suppose that one of $j$ and $k$, say $j$, equals to one. In this case $S$ has at least four vertices in $\{g_1,g_2,g_3\}\times \{h_1,h_2,h_3\}$,
which have no influence on vertices in $\{g_4,\ldots,g_r\}\times \{h_4,\ldots,h_t\}$ which induce $K_{r-3}\times K_{t-3}$. By the induction
hypothesis we get $|S|\geq 4+dim(K_{r-3}\times K_{t-3})$ and the lower bound follows in this final case.
\end{proof}

Since $C_3\cong K_3$, the theorem above already gives a first part of the answer to $dim(C_r\times K_t)$.
The next proposition completes this family of direct products.

\begin{proposition}\label{C-r-K-n}
For any integers $t\ge 3$ and $r\ge 4$,
$$dim(C_r\times K_t)=\left\lceil\frac{r}{3}\right\rceil(t-1).
$$
\end{proposition}

\begin{proof}
Let $U=\{g_0,\ldots,g_{r-1}\}$ and $V=\{h_1,\ldots,h_t\}$ be the vertex sets of $C_r$ and $K_t$, respectively.
From now on, in this proof, all the operations with the subindexes of $g_i$ are done modulo $r$.  First, let $r\ne 6$. If $A=\{g_i:i\,\equiv\;1\;\mbox{ (mod }\,3)\}$, then the set $B=A\times(V-\{h_1\})$ is clearly a metric generator for $C_r\times K_t$. Therefore,
$dim(C_r\times K_t)\le \left\lceil\frac{r}{3}\right\rceil(t-1)$.

Let now $r=6$. Notice that the set $B$ defined as above is not a metric generator since, for instance, $(g_3,h_1)$ and $(g_5,h_1)$
are not distinguished by $B$. However, it is easy to observe that
$$W=\{(g_i,h_i), (g_{i+3},h_i):1\leq i\leq t-2\})\cup\{(g_{t-1},h_{t-1}),(g_{t+2},h_{t})\}$$
is a metric generator for $C_6\times K_t$. Therefore,
$$dim(C_6\times K_t)\le 2(t-2)+2=\left\lceil\frac{6}{3}\right\rceil(t-1).$$

On the other hand, notice that for every $i\in \{0,\ldots,r-1\}$, two different vertices $(g_i,h_j),(g_i,h_{\ell})$ are
only distinguished by themselves and by the vertices $(g_{i-1},h_j),(g_{i-1},h_{\ell}),(g_{i+1},h_j),(g_{i+1},h_{\ell})$.
Thus, if $S$ is a metric basis for $C_r\times K_t$ and  $S_i=S\cap (\{g_{i-1},g_i,g_{i+1}\}\times V)$
for every $i\in \{0,\ldots,r-1\}$, then it follows that $|S_i|\ge t-1$. We have that
$$dim(C_r\times K_t)=|S|=\frac{1}{3}\sum_{i=1}^{r}|S_i|\ge \frac{1}{3}\sum_{i=1}^{r}(t-1)=\frac{r}{3}(t-1),$$
which leads to $dim(C_r\times K_t)\ge \left\lceil\frac{r}{3}\right\rceil(t-1)$ and the equality follows for
$r\equiv\,0\,\mbox{ (mod }\,3)$.

Now, assume $r\equiv\,1\,\mbox{ (mod }\,3)$ or $r\equiv\,2\,\mbox{ (mod }\,3)$. If
$|S|<\left\lceil\frac{r}{3}\right\rceil(t-1)$, then there exists at least one vertex $g_j\in U$ such that
$|S_j|<t-1$. Thus, there are two vertices $(g_j,h_e),(g_j,h_f)$ which are not distinguished by $S$, a
contradiction. Thus $|S|\ge \left\lceil\frac{r}{3}\right\rceil(t-1)$ and we are done.
\end{proof}

A similar procedure as the one above, in the case $r\ne 6$, gives the following result. However, we include its
proof, since there are some different details in the process.

\begin{proposition}
For any integers $r,t\ge 3$,
$$dim(P_r\times K_t)=\left\lceil\frac{r}{3}\right\rceil(t-1).$$
\end{proposition}

\begin{proof}
Let $U=\{g_1,\ldots,g_{r}\}$ and $V=\{h_1,\ldots,h_t\}$ be the vertex sets of $P_r$ and $K_t$,
respectively. Similarly to the proof of Proposition \ref{C-r-K-n}, if
$A=\{g_i:i\,\equiv\;2\;\mbox{ (mod }\,3)\}$ when $r\equiv\,0\,\mbox{ (mod }\,3)$, or
$A=\{g_i:i\,\equiv\;1\;\mbox{ (mod }\,3)\}$ otherwise, then the set $B=A\times(V-\{h_1\})$ is
clearly a metric generator for $P_r\times K_t$. Therefore,
$dim(P_r\times K_t)\le \left\lceil\frac{r}{3}\right\rceil(t-1)$.

On the other hand, again as in the proof of Proposition \ref{C-r-K-n}, we notice that for every
$i\in \{2,\ldots,r-1\}$, two different vertices $(g_i,h_j),(g_i,h_{\ell})$ are only distinguished by
themselves and the vertices $(g_{i-1},h_j),(g_{i-1},h_{\ell}),(g_{i+1},h_j),(g_{i+1},h_{\ell})$. We first
consider $r\equiv\,0\,\mbox{ (mod }3)$. Let $S$ be a metric basis for $C_r\times K_t$ and let
$S_i=S\cap (\{g_{i-1},g_i,g_{i+1}\}\times V)$ for every $i\equiv\,2\,\mbox{ (mod }3)$. Hence, it
follows that $|S_i|\ge t-1$ and we have that
$$dim(C_r\times K_t)=|S|=\sum_{i=0}^{\left\lceil\frac{r}{3}\right\rceil-1}|S_{3i+2}|\ge
\sum_{i=0}^{\left\lceil\frac{r}{3}\right\rceil-1}(t-1)=\left\lceil\frac{r}{3}\right\rceil(t-1),$$
Thus, the equality follows for $r\equiv\,0\,\mbox{ (mod }\,3)$.

Now, if $r\equiv\,1\,\mbox{ (mod }\,3)$ or $r\equiv\,2\,\mbox{ (mod }\,3)$, then suppose that
$|S|<\left\lceil\frac{r}{3}\right\rceil(t-1)$. So, there exists at least one vertex $g_j\in U$
such that $|S_j|<t-1$. Thus, there are two vertices $(g_j,h_e),(g_j,h_f)$ which are not
distinguished by $S$, a contradiction. Therefore, $|S|=\left\lceil\frac{r}{3}\right\rceil(t-1)$,
which completes the proof.
\end{proof}

It is known from \cite{direct-cart-isom} that $G\Box H\cong G\times H$ if and only if
$G\cong H\cong C_{2k+1}$ for some positive integer $k$. The following formula
obtained, for the metric dimension of $C_r\Box C_t$, $r,t\ge 3$, is known from \cite{pelayo1}:
$$dim(C_r\Box C_t)=\left\{\begin{array}{ll}
                            3, & \mbox{if $r$ or $t$ is odd,} \\
                            4, & \mbox{otherwise.}
                          \end{array}
\right.$$
Thus, by using these two facts we obtain the next result.

\begin{corollary}
For any integer $k\ge 1$, $dim(C_{2k+1}\times C_{2k+1})=3$.
\end{corollary}


\section{Strong metric dimension}

We divide this section into two subsections. The first one deals with the strong metric dimension of the direct
product of any graph $G$ with a complete graph $K_n$, and in the second one we analyze some cases
of direct products between two graphs of diameter two. For this, note that by $G\cup H$ we mean a graph
with $V(G\cup H)=V(G)\cup V(H)$ and $E(G\cup H)=E(G)\cup E(H)$. It is easy to see that twins are
mutually maximally distant vertices with distance two between them.

\subsection{$H\cong K_n$}

We start to describe the structure of the strong resolving graph of $G\times K_n$ in order to use
Theorem \ref{lem_oellerman}. A graph $G$ is 2-\emph{mutually maximally distant free} or 2MMF for short,
if there exists no pair of mutually maximally distant vertices $u$ and $v$ with $d_G(u,v)=2$.
Clearly diameter two graphs are not 2MMF graphs.

\begin{theorem}\label{direct-complete}
Let $G$ be a connected 2MMF graph of order at least three and let the integer $n\ge 3$. If
$W$ is the subset of $V(G)$ which contains all vertices belonging to a triangle in $G$, then
$$(G\times K_n)_{SR}\cong (G\Box N_n)\cup (G_{SR}\circ N_n)\cup (W\Box K_n).$$
\end{theorem}

\begin{proof}
Let $(g_1,h_1),(g_2,h_2)$ be two different vertices of $G\times K_n$. We first consider $G$
is a triangle free graph and analyze the following possible situations.\\

\noindent Case 1: $g_1\ne g_2$, $h_1=h_2$ and $g_1\sim g_2$. Hence
$d_{G\times K_n}((g_1,h_1),(g_2,h_1))=3$, since $G$ is triangle free. Also, we observe that
$N_{G\times K_n}(g_1,h_1)=N_G(g_1)\times (V(K_n)-\{h_1\})$ and for every vertex $g\in N_G(g_1)$
and every $h\in V(K_n)-\{h_1\}$ it follows, $d_{G\times K_n}((g_2,h_1),(g,h))=2$. Similarly,
$N_{G\times K_n}(g_2,h_1)=N_G(g_2)\times (V(K_n)-\{h_1\})$ and for every vertex $g\in N_G(g_2)$
and every $h\in V(K_n)-\{h_1\}$ it follows, $d_{G\times K_n}((g_1,h_1),(g,h))=2$. Thus,
$(g_1,h_1)$ and $(g_2,h_2)$ are mutually maximally distant in $G\times K_n$.\\

As a consequence of the Case 1 above, for any vertex $h\in V(K_n)$ and any two adjacent vertices
$g,g'$ of $G$, it follows that $(g,h)$ and $(g',h)$ are mutually maximally distant in $G\times K_n$.
Therefore, the strong resolving graph $(G\times K_n)_{SR}$ contains $n$ copies of $G$ as subgraphs,
or equivalently the graph $G\Box N_n$. We continue describing the other part of $(G\times K_n)_{SR}$.\\

\noindent Case 2: $g_1\ne g_2$, $g_1\not\sim g_2$ and $g_1,g_2$ are mutually maximally distant in $G$.
Hence, it follows by (\ref{distance}) that $d_{G\times K_n}((g_1,h_1),(g_2,h_2))=d_G(g_1,g_2)\geq 3$, since
$G$ is 2MMF graph. It is straightforward
to observe that $(g_1,h_1)$ and $(g_2,h_2)$ are mutually maximally distant in $G\times K_n$. \\

As a consequence of the Case 2 above, for any vertices $h,h'\in V(K_n)$ and any two mutually maximally
distant vertices $g,g'$ of $G$, it follows that $(g,h)$ and $(g',h')$ are mutually maximally distant
in $G\times K_n$. Therefore, the strong resolving graph $(G\times K_n)_{SR}$ contains a subgraph
isomorphic to the lexicographic product of $G_{SR}$ and $N_n$. We now shall show that $(G\times K_n)_{SR}$
has no more edges than those ones described until now, which leads to
$(G\times K_n)_{SR}\cong (G\Box N_n)\cup (G_{SR}\circ N_n)$.\\

\noindent Case 3: $g_1\ne g_2$, $g_1\not\sim g_2$ and $g_1,g_2$ are not mutually maximally distant
in $G$. Similarly to Case 2, it clearly follows that $(g_1,h_1)$ and $(g_2,h_2)$ are not mutually
maximally distant in $G\times K_n$, since for a neighbor $g_3$ of $g_2$ with $d_{G}(g_1,g_3)>d_{G}(g_1,g_2)$
we obtain $d_{G\times K_n}((g_1,h_1),(g_2,h_2))<d_{G\times K_n}((g_1,h_1),(g_3,h))$ for any $h\ne h_2$.\\

\noindent Case 4: $g_1\ne g_2$, $h_1\ne h_2$ and $g_1\sim g_2$. Hence
$d_{G\times K_n}((g_1,h_1),(g_2,h_2))=1$. Since $n\ge 3$, for any vertex $h_3\notin \{h_1,h_2\}$ we have
that $(g_1,h_3)\in N_{G\times K_n}(g_2,h_2)$ and $d_{G\times K_n}((g_1,h_1),(g_1,h_3))=2$. Thus,
$(g_1,h_1)$ and $(g_2,h_2)$ are not mutually maximally distant in $G\times K_n$.\\

\noindent Case 5: $g_1=g_2$. Hence, $d_{G\times K_n}((g_1,h_1),(g_1,h_2))=2$. Since $G$ has order
greater than one, there exists a vertex $g_3\in N_G(g_1)$ and we observe that the vertex
$(g_3,h_1)\in N_{G\times K_n}(g_1,h_2)$. Also, as $G$ is triangle free,
$d_{G\times K_n}((g_1,h_1),(g_3,h_1))=3$. Thus, $(g_1,h_1)$ and $(g_2,h_2)$ are not mutually
maximally distant in $G\times K_n$.\\

So, if $G$ is triangle free, then we have that $(G\times K_n)_{SR}\cong (G\Box N_n)\cup (G_{SR}\circ N_n)$.
We consider now that $W$ is the set of vertices of $G$ belonging to a triangle and $|W|=t$. We notice
that the fact that there exist vertices belonging to a triangle in $G$ only affects the Case 5. (Actually it
also affect Case 1, but there are no changes in conclusions.) That
is, if $g_1=g_2$ and $g_1\in W$, then as above $d_{G\times K_n}((g_1,h_1),(g_1,h_2))=2$. However,
we have that $N_{G\times K_n}(g_1,h_1)=N_G(g_1)\times (V(K_n)-\{h_1\})$ and for every vertex
$g\in N_G(g_1)$ and every $h\in V(K_n)-\{h_1\}$ it follows, $d_{G\times K_n}((g_2,h_2),(g,h))\le 2$.
Similarly, $N_{G\times K_n}(g_2,h_2)=N_G(g_2)\times (V(K_n)-\{h_2\})$ and for every vertex $g\in N_G(g_2)$
and every $h\in V(K_n)-\{h_2\}$ it follows, $d_{G\times K_n}((g_1,h_1),(g,h))\le 2$.  Thus, $(g_1,h_1)$ and
$(g_1,h_2)$ are mutually maximally distant in $G\times K_n$.\\

As a consequence, given a vertex $g\in W$, for any two vertices $h,h'\in V(K_n)$ it follows that $(g,h)$
and $(g,h')$ are mutually maximally distant in $G\times K_n$. Therefore, the strong resolving graph
$(G\times K_n)_{SR}$ contains $t=|W|$ copies of the graph $K_n$, or equivalently the Cartesian product
of $W$ and $K_n$ and the proof is completed.
\end{proof}

Now, according to Theorems \ref{lem_oellerman} and \ref{direct-complete}, the study of the strong metric dimension
of $G\times K_n$ is reduced to study the vertex cover number of $(G\Box N_n)\cup (G_{SR}\circ N_n)\cup (W\Box K_n)$,
where $W$ contains all vertices of $G$ which are on some triangle.
Next we show that this can be even more reduced for the case of triangle free graphs, which are
one extreme with respect to $W$.

\begin{lemma}\label{lem-GSR-trian-free}
For any triangle free 2MMF graph $G$ of order at least three and any integer $n\ge 3$,
$$\beta((G\Box N_n)\cup (G_{SR}\circ N_n))=n\cdot \beta(G\cup G_{SR}).$$
\end{lemma}

\begin{proof}
Since $(G\cup G_{SR})\Box N_n$ is a subgraph of $(G\Box N_n)\cup (G_{SR}\circ N_n)$, it is clear that
$n\cdot\beta(G\cup G_{SR})=\beta((G\cup G_{SR})\Box N_n)\le \beta((G\Box N_n)\cup (G_{SR}\circ N_n))$.
On the other hand, let $A_i$, $i\in \{1,\ldots,n\}$, be a vertex cover set of minimum cardinality in the
$i^{th}$ copy of $G\cup G_{SR}$ in $(G\cup G_{SR})\Box N_n$. Let $e=xy$ be an edge from
$(G\Box N_n)\cup (G_{SR}\circ N_n)$ between any two
vertices belonging to two different copies, say the $i^{th}$ and the $j^{th}$ copies, of $G\cup G_{SR}$
in $(G\cup G_{SR})\Box N_n$. Thus, the vertices $x$ and $y$ are mutually maximally distant in $G$, which
means that $x\in A_i$ or $y\in A_j$. Thus, every edge $e$ of $(G\Box N_n)\cup (G_{SR}\circ N_n)$ is
covered by $\bigcup_{i=1}^n A_i$ and, as a consequence, we obtain that
$\beta((G\Box N_n)\cup (G_{SR}\circ N_n))\le \sum_{i=1}^n|A_i|= \sum_{i=1}^n \beta(G\cup G_{SR})=n\cdot \beta(G\cup G_{SR})$,
which completes the proof.
\end{proof}

The following result follows directly from Theorems \ref{lem_oellerman} and \ref{direct-complete},
and from Lemma \ref{lem-GSR-trian-free}.

\begin{theorem}\label{th-dim_s-G-K_n}
Let $G$ be a connected 2MMF graph of order at least three and let $n\ge 3$ be an integer. If $W$
contains all vertices of $G$ belonging to a triangle in $G$, then
$$dim_s(G\times K_n)= \beta((G\Box N_n)\cup (G_{SR}\circ N_n)\cup (W\Box K_n)).$$
Moreover, if $G$ is triangle free, then $dim_s(G\times K_n)=n\cdot \beta(G\cup G_{SR})$,
and if every vertex of $G$ is in a triangle, then $dim_s(G\times K_n)= \beta((G\Box K_n)\cup (G_{SR}\circ N_n))$.
\end{theorem}

Trees are graphs without triangles. However they are not always 2MMF graphs. Given a tree $T$ of order at least
three, we denote by $T_{-\ell}$ the tree obtained from $T$ by deleting all its leaves. A vertex of $T$ is called
a \emph{support vertex}, if it is adjacent to a leaf. Clearly, any tree $T$ is a 2MMF graph if and only if every support vertex
is adjacent to exactly one leaf. Moreover, if there exists a $\beta(T_{-\ell})$-set that contains a support
vertex of $T$, then we call $T$ a \emph{good tree}.

\begin{proposition}\label{tree}
Let $T$ be a 2MMF tree with $\ell(T)$ leaves and $n\ge 3$ be an integer.
\begin{itemize}
\item If $T$ is a good tree, then $dim_s(T\times K_n)=n(\ell(T)-1+\beta(T_{-\ell}))$,
\item If $T$ is not a good tree, then $dim_s(T\times K_n)=n(\ell(T)+\beta(T_{-\ell}))$.
\end{itemize}
\end{proposition}

\begin{proof}
Since $T$ is triangle free, $T\cup T_{SR}$ is isomorphic to a graph obtained from $T$ by adding all
possible edges between any two leaves. So, leaves of $T$ induces a complete graph in $T\cup T_{SR}$ and
every $\beta(T\cup T_{SR})$-set contains at least $n-1$ leaves. If $T$ is a good tree, then there
exists a $\beta(T_{-\ell})$-set $A$ that contains a support vertex $v$ of $T$. Let $u$ be a leave adjacent to
$v$. If $B$ is the set of all leaves of $T$, then $A\cup (B-\{u\})$ is clearly a $\beta(T\cup T_{SR})$-set and we have
$dim_s(T\times K_n)=n(\ell(T)-1+\beta(T_{-\ell}))$ by Theorem \ref{th-dim_s-G-K_n}.

If $T$ is not a good tree, then there is no support vertex of $T$ in every $\beta(T_{-\ell})$-set. In this case we need
at least $\ell(T)$ additional vertices for a vertex covering of $T\cup T_{SR}$. Since all leaves together with
a $\beta(T_{-\ell})$-set form a vertex covering set of $T\cup T_{SR}$, the second equality follows
by Theorem \ref{th-dim_s-G-K_n} again.
\end{proof}

In the next result a subdivided star represents a tree obtained from a star by subdividing once each of its edges.

\begin{corollary}
Let $T$ be a tree of order $n_1\geq 4$ with $\ell(T)$ leaves and let $n\ge 3$ be an integer.
\begin{itemize}
\item If $T$ is a path $P_{n_1}$, then $dim_s(P_{n_1}\times K_n)=n\left\lceil\frac{n_1}{2}\right\rceil$.
\item If $T$ is a subdivided star, then $dim_s(T\times K_n)=n(\frac{n_1+1}{2})$.
\end{itemize}
\end{corollary}

\begin{proof}
The results follow directly from Proposition \ref{tree} and the fact that a path on odd vertices and a subdivided star are not good trees, while a path on even vertices is a good tree.
\end{proof}

Note that a result above regarding a path $P_{n_1}$ was already presented in \cite{str-dim-cart-dir}. We next deal with the direct product of a complete graph and a complete bipartite graph. In contrast with Theorem \ref{direct-complete}, in this case all the mutually maximally distant vertices of the complete bipartite graph are at distance two.

\begin{theorem}\label{lem-K_n-K_r-t}
For any $r,t\ge 1$ and any $n\ge 3$, $$(K_{r,t}\times K_n)_{SR}\cong \bigcup_{i=1}^n K_{r+t}.$$
\end{theorem}

\begin{proof}
Let $X,Y$ be the bipartition sets of $K_{r,t}$ such that $|X|=r$ and $|Y|=t$. Consider the vertices $g\in X$ and $h\in V(K_n)$. We notice that vertices in $A=Y\times (V(K_n)-\{h\})$ form the open neighborhood of $(g,h)$. Since $n\ge 3$, every vertex from $(X\times V(K_n))-\{(g,h)\})$ has a neighbor in $A$ and viceversa. Thus, vertices of $A$ are not mutually maximally distant with $(g,h)$. On the other hand, the remaining vertices are $Y\times \{h\}$ and they are adjacent to all vertices in $X\times (V(K_n)-\{h\})$. Clearly, any vertex in $Y\times \{h\}$ is mutually maximally distant with $(g,h)$. Moreover, the vertices in $X\times (V(K_n)-\{h\})$ are not mutually maximally distant with $(g,h)$. Finally, we notice that the vertices in $(X-\{g\})\times \{h\}$ are not adjacent to any vertex in $Y\times \{h\}$. So, every vertex in $(X-\{g\})\times \{h\}$ is mutually maximally distant with $(g,h)$. As a consequence, $(g,h)$ is adjacent in $(K_{r,t}\times K_n)_{SR}$ to every vertex of $(V(K_{r,t})-\{g\})\times \{h\}$. Therefore, by symmetry, the proof is completed.
\end{proof}

Our next result is a consequence of Theorems \ref{lem_oellerman} and \ref{lem-K_n-K_r-t}.

\begin{theorem}
For any $r,t\ge 1$ and any $n\ge 3$, $$dim_s(K_{r,t}\times K_n)=n(r+t-1).$$
\end{theorem}


\subsection{Diameter two graphs}

As we have seen, the complete graphs as a factor of a direct product provide a various palette of results for strong
metric dimension. The natural extension of them are graphs of diameter two and we present some results for them in
this last part.
Since we need to be careful with connectedness of the direct product, we separate the results with respect
to whether one factor is bipartite or not. It is not hard to see that the only bipartite graphs of diameter two are the complete
bipartite graphs $K_{k,\ell}$, where $\max \{k,\ell\}\geq 2$.

Another important measure for the strong metric dimension of a direct product of two graphs of diameter two is when the
factors are triangle free and moreover, when every pair of vertices is on a five-cycle. Hence, we call a graph
in which every pair of vertices is on a common five-cycle a $C_5$-\emph{connected} graph. Clearly, a $C_5$-connected
graph has diameter at most two. Moreover, if $G$ is a triangle free $C_5$-connected graph, then its diameter equals two.
The Petersen graph is $C_5$-connected triangle free graph. The graph $G$ of Figure \ref{fig:C6} is an example of a triangle free graph of diameter
two in which $u$ and $v$ are not on a common five-cycle and $G$ is not $C_5$-connected. The graph $H$ of the same figure is a
triangle free $C_5$-connected graph of diameter two.

\begin{figure}[ht!]
\begin{center}
\begin{tikzpicture}[scale=0.6,style=thick]
\def\vr{4pt}
\path (5,-3) coordinate (a); \path (8,0) coordinate (b);
\path (5,3) coordinate (c); \path (2,0) coordinate (d);
\path (4,0) coordinate (e); \path (5,1) coordinate (f);
\path (6,0) coordinate (g); \path (5,-1) coordinate (h);

\path (-3,-2) coordinate (v); \path (-4,0.5) coordinate (w);
\path (-1.5,0.5) coordinate (x); \path (-4,3) coordinate (y);
\path (-5,-2) coordinate (u); \path (-6.5,0.5) coordinate (z);

\draw (a) -- (b); \draw (g) -- (h);
\draw (c) -- (b); \draw (c) -- (f);
\draw (b) -- (g); \draw (a) -- (h);
\draw (a) -- (d); \draw (c) -- (d);
\draw (d) -- (e); \draw (f) -- (e);
\draw (g) -- (e); \draw (f) -- (h);

\draw (z) -- (u);
\draw (u) -- (v); \draw (u) -- (w);
\draw (v) -- (x); \draw (x) -- (y);
\draw (y) -- (z); \draw (y) -- (w);
\draw (a)  [fill=black] circle (\vr); \draw (b)  [fill=black] circle (\vr);
\draw (c)  [fill=black] circle (\vr); \draw (d)  [fill=black] circle (\vr);
\draw (e)  [fill=black] circle (\vr); \draw (f)  [fill=black] circle (\vr);
\draw (g)  [fill=black] circle (\vr); \draw (h)  [fill=black] circle (\vr);
\draw (u)  [fill=black] circle (\vr); \draw (v)  [fill=black] circle (\vr);
\draw (w)  [fill=black] circle (\vr); \draw (x)  [fill=black] circle (\vr);
\draw (y)  [fill=black] circle (\vr); \draw (z)  [fill=black] circle (\vr);
\draw (2,3) node {$H$}; \draw (-6,3) node {$G$};
\draw (-7,0.5) node {$u$}; \draw (-3.5,0.5) node {$v$};
\end{tikzpicture}
\end{center}
\caption{Two triangle free graphs of diameter two.} \label{fig:C6}
\end{figure}
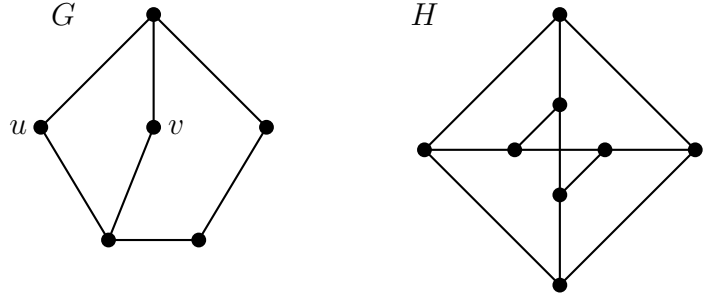

\begin{theorem}\label{complbip}
Let $G$ be a nonbipartite triangle free graph of order $n\ge 2$ and let $\max\{k,\ell\}\geq 2$. If $G$ is $C_5$-connected, then
$$dim_s(G\times K_{k,\ell})=n(k+\ell -1).$$
\end{theorem}

\begin{proof}
In order to use Theorem \ref{lem_oellerman} we first describe $(G\times K_{k,\ell})_{SR}$. Let $V(G)=\{g_1,\ldots ,g_n\}$
and $U(K_{k,\ell})=U_1\cup U_2$ where $U_1=\{u_1,\ldots ,u_k\}$ and $U_2=\{v_1,\ldots ,v_{\ell}\}$.
Clearly, $d^e_{K_{k,\ell}}(u_i,v_j)=\infty$, $d^o_{K_{k,\ell}}(u_i,v_j)=1$, $d^o_{K_{k,\ell}}(v_i,v_j)=\infty$ and
$d^o_{K_{k,\ell}}(u_i,u_j)=\infty$ for any $i$ and $j$. Also, $d^e_{K_{k,\ell}}(u_i,u_j)=2$ and $d^e_{K_{k,\ell}}(v_i,v_j)=2$
for every $i\neq j$. Conversely, by $C_5$-connectedness of $G$, $d^e_G(g_i,g_j)$ and $d^o_G(g_i,g_j)$ always exists.
Moreover, $d^e_G(g_i,g_j)$ is between 0 and 4, while $d^o_G(g_i,g_j)$ is between 1 and 5. Hence, by the distance formula
(\ref{distance}) we can have the distances between 0 and 5 in $G\times K_{k,\ell}$.
Again, by this distance formula, it is easy to see that $d_{G\times K_{k,\ell}}((g_1,u_1),(g_1,v_j))=5$ for any $j\in \{1,\ldots ,\ell\}$
and that $d_{G\times K_{k,\ell}}((g_1,u_1),(g_1,u_j))=2$ for any $j\in \{2,\ldots ,k\}$. We show that these are the only neighbors
of $(g_1,u_1)$ in $(G\times K_{k,\ell})_{SR}$. Clearly, $(g_1,u_1)$ and $(g_1,v_j)$ are mutually maximally distant, since they are
diametral vertices for any $j\in \{1,\ldots ,\ell\}$. Since $N_{K_{k,\ell}}(u_1)=N_{K_{k,\ell}}(u_j)$, for any $j\in \{2,\ldots ,k\}$,
by using (\ref{neigh}), we see that $(g_1,u_1)$ and $(g_1,u_j)$ have the same neighborhood and therefore, they are mutually maximally distant.

Next we show that no other vertex of $G\times K_{k,\ell}$ is mutually maximally distant with $(g_1,u_1)$. In this case, we reduce it
to a five-cycle, since $G$ is $C_5$-connected. We may assume that $g_1g_2g_3g_4g_5g_1$ is a five-cycle. By the symmetry of a five-cycle
we need to present the arguments only for $g_2$ and $g_3$. For every $j\in \{1,\ldots ,\ell\}$ and $i\in \{2,\ldots ,\ell\}$ they are as follows:
\begin{itemize}
\item $(g_2,v_j)\sim (g_3,u_1)$ and $(g_2,v_j)$ is closer to $(g_1,u_1)$ than $(g_3,u_1)$;
\item $(g_2,u_i)\sim (g_1,v_1)$ and $(g_2,u_i)$ is closer to $(g_1,u_1)$ than $(g_1,v_1)$;
\item $(g_3,v_j)\sim (g_2,u_1)$ and $(g_3,v_j)$ is closer to $(g_1,u_1)$ than $(g_2,u_1)$;
\item $(g_3,u_i)\sim (g_4,v_1)$ and $(g_3,u_i)$ is closer to $(g_1,u_1)$ than $(g_4,v_1)$.
\end{itemize}
See the graph $C_5\times K_{1,2}\cong C_5\times P_3$ on the left part of Figure \ref{second}, where the distances from $(g_1,u_1)$ are marked.
Thus, the vertex $(g_1,u_1)$ is adjacent to all vertices of $\{g_1\}\times (V(K_{k,\ell})-\{u_1\})$ in $(G\times K_{k,\ell})_{SR}$.
(Notice that the same argument hold also when $\min\{k,\ell\}=1$.)
We can use the same arguments for any vertex of $G\times K_{k,\ell}$ and therefore, we have $(G\times K_{k,\ell})_{SR}\cong N_n\Box K_{k+\ell}$.
By Theorem \ref{lem_oellerman} we have that $dim_s(G\times K_{k,\ell})=\beta ((G\times K_{k,\ell})_{SR})=n\beta(K_{k+\ell})=n(k+\ell -1)$
and the proof is completed.
\end{proof}


\begin{figure}[ht!]
\begin{center}
\begin{tikzpicture}[scale=0.6,style=thick]
\def\vr{4pt}
\path (2,-4) coordinate (a1); \path (4,-4) coordinate (a2);
\path (6,-4) coordinate (a3); \path (8,-4) coordinate (a4);
\path (10,-4) coordinate (a5); \path (2,-2) coordinate (b1);
\path (4,-2) coordinate (b2); \path (6,-2) coordinate (b3);
\path (8,-2) coordinate (b4); \path (10,-2) coordinate (b5);
\path (2,0) coordinate (c1); \path (4,0) coordinate (c2);
\path (6,0) coordinate (c3); \path (8,0) coordinate (c4);
\path (10,0) coordinate (c5); \path (2,2) coordinate (d1);
\path (4,2) coordinate (d2); \path (6,2) coordinate (d3);
\path (8,2) coordinate (d4); \path (10,2) coordinate (d5);
\path (2,4) coordinate (e1); \path (4,4) coordinate (e2);
\path (6,4) coordinate (e3); \path (8,4) coordinate (e4);
\path (10,4) coordinate (e5);
\path (6.2,-0.7) coordinate (x); \path (6.2,0.7) coordinate (y);

\path (-10,-3) coordinate (u1); \path (-8,-3) coordinate (u2);
\path (-6,-3) coordinate (u3); \path (-4,-3) coordinate (u4);
\path (-2,-3) coordinate (u5); \path (-10,0) coordinate (v1);
\path (-8,0) coordinate (v2); \path (-6,0) coordinate (v3);
\path (-4,0) coordinate (v4); \path (-2,0) coordinate (v5);
\path (-10,3) coordinate (w1); \path (-8,3) coordinate (w2);
\path (-6,3) coordinate (w3); \path (-4,3) coordinate (w4);
\path (-2,3) coordinate (w5);

\draw (a1) -- (b2); \draw (a1) -- (e2);
\draw (a1) -- (b5); 
\draw (a2) -- (b1); \draw (a2) -- (e1);
\draw (a2) -- (b3); \draw (a2) -- (e3);
\draw (a3) -- (b2); \draw (a3) -- (e2);
\draw (a3) -- (b4); \draw (a3) -- (e4);
\draw (a4) -- (b3); \draw (a4) -- (e3);
\draw (a4) -- (b5); \draw (a4) -- (e5);
\draw (a5) -- (b1); 
\draw (a5) -- (b4); \draw (a5) -- (e4);
\draw (b1) -- (c2); \draw (b1) -- (c5);
\draw (b2) -- (c1); \draw (b2) -- (c3);
\draw (b3) -- (c2); \draw (b3) -- (c4);
\draw (b4) -- (c3); \draw (b4) -- (c5);
\draw (b5) -- (c1); \draw (b5) -- (c4);
\draw (c1) -- (d2); \draw (c1) -- (d5);
\draw (c2) -- (d1); \draw (c2) -- (d3);
\draw (c3) -- (d2); \draw (c3) -- (d4);
\draw (c4) -- (d3); \draw (c4) -- (d5);
\draw (c5) -- (d1); \draw (c5) -- (d4);
\draw (d1) -- (e2); \draw (d1) -- (e5);
\draw (d2) -- (e1); \draw (d2) -- (e3);
\draw (d3) -- (e2); \draw (d3) -- (e4);
\draw (d4) -- (e3); \draw (d4) -- (e5);
\draw (d5) -- (e1); \draw (d5) -- (e4);
\draw (e5) -- (x); \draw (x) -- (a1);
\draw (y) -- (e1); \draw (a5) -- (y);

\draw (u1) -- (v2); \draw (u1) -- (v5);
\draw (u2) -- (v1); \draw (u2) -- (v3);
\draw (u3) -- (v2); \draw (u3) -- (v4);
\draw (u4) -- (v3); \draw (u4) -- (v5);
\draw (u5) -- (v4); \draw (u5) -- (v1);
\draw (v1) -- (w2); \draw (v1) -- (w5);
\draw (v2) -- (w3); \draw (v2) -- (w1);
\draw (v3) -- (w2); \draw (v3) -- (w4);
\draw (v4) -- (w3); \draw (v4) -- (w5);
\draw (v5) -- (w1); \draw (v5) -- (w4);
\draw (a1)  [fill=black] circle (\vr); \draw (a2)  [fill=black] circle (\vr);
\draw (a3)  [fill=black] circle (\vr); \draw (a4)  [fill=black] circle (\vr);
\draw (a5)  [fill=black] circle (\vr); \draw (b1)  [fill=black] circle (\vr);
\draw (b2)  [fill=black] circle (\vr); \draw (b3)  [fill=black] circle (\vr);
\draw (b4)  [fill=black] circle (\vr); \draw (b5)  [fill=black] circle (\vr);
\draw (c1)  [fill=black] circle (\vr); \draw (c2)  [fill=black] circle (\vr);
\draw (c3)  [fill=black] circle (\vr); \draw (c4)  [fill=black] circle (\vr);
\draw (c5)  [fill=black] circle (\vr); \draw (d1)  [fill=black] circle (\vr);
\draw (d2)  [fill=black] circle (\vr); \draw (d3)  [fill=black] circle (\vr);
\draw (d4)  [fill=black] circle (\vr); \draw (d5)  [fill=black] circle (\vr);
\draw (e1)  [fill=black] circle (\vr); \draw (e2)  [fill=black] circle (\vr);
\draw (e3)  [fill=black] circle (\vr); \draw (e4)  [fill=black] circle (\vr);
\draw (e5)  [fill=black] circle (\vr);

\draw (u1)  [fill=black] circle (\vr);
\draw (u2)  [fill=black] circle (\vr); \draw (u3)  [fill=black] circle (\vr);
\draw (u4)  [fill=black] circle (\vr); \draw (u5)  [fill=black] circle (\vr);
\draw (v1)  [fill=black] circle (\vr); \draw (v2)  [fill=black] circle (\vr);
\draw (v3)  [fill=black] circle (\vr); \draw (v4)  [fill=black] circle (\vr);
\draw (v5)  [fill=black] circle (\vr); \draw (w1)  [fill=black] circle (\vr);
\draw (w2)  [fill=black] circle (\vr); \draw (w3)  [fill=black] circle (\vr);
\draw (w4)  [fill=black] circle (\vr); \draw (w5)  [fill=black] circle (\vr);

\draw (-10.5,-3.3) node {$0$}; \draw (-7.5,-3.3) node {$4$};
\draw (-5.5,-3.3) node {$2$}; \draw (-3.5,-3.3) node {$2$};
\draw (-1.5,-3.3) node {$4$}; \draw (-10.5,0) node {$5$};
\draw (-7.5,0) node {$1$}; \draw (-5.5,0) node {$3$};
\draw (-3.5,0) node {$3$}; \draw (-1.5,0) node {$1$};
\draw (-10.5,3.3) node {$2$}; \draw (-7.5,3.3) node {$4$};
\draw (-5.5,3.3) node {$2$}; \draw (-3.3,3.3) node {$2$};
\draw (-1.5,3.3) node {$4$}; \draw (-10,-5) node {$g_1$};
\draw (-8,-5) node {$g_2$}; \draw (-6,-5) node {$g_3$};
\draw (-4,-5) node {$g_4$}; \draw (-2,-5) node {$g_5$};
\draw (-11.5,-3) node {$u_1$}; \draw (-11.3,3) node {$u_2$};
\draw (-11.5,0) node {$v_1$}; 

\draw (1.5,-4.3) node {$0$}; \draw (4.5,-4.3) node {$4$};
\draw (6.5,-4.3) node {$2$}; \draw (8.5,-4.3) node {$2$};
\draw (10.5,-4.3) node {$4$}; \draw (1.5,-2) node {$4$};
\draw (4.4,-2) node {$1$}; \draw (6.5,-2) node {$3$};
\draw (8.5,-2) node {$3$}; \draw (10.5,-2) node {$1$};
\draw (1.5,0) node {$2$}; \draw (4.5,0) node {$3$};
\draw (6.5,0) node {$2$}; \draw (8.5,0) node {$2$};
\draw (10.5,0) node {$3$}; \draw (2,-5.5) node {$g_1$};
\draw (4,-5.5) node {$g_2$}; \draw (6,-5.5) node {$g_3$};
\draw (8,-5.5) node {$g_4$}; \draw (10,-5.5) node {$g_5$};
\draw (0.5,-4) node {$h_1$}; \draw (0.5,-2) node {$h_2$};
\draw (0.5,0) node {$h_3$}; \draw (0.5,2) node {$h_4$};
\draw (0.5,4) node {$h_5$}; \draw (1.5,2) node {$2$};
\draw (4.5,2) node {$3$}; \draw (6.5,2) node {$2$};
\draw (8.5,2) node {$2$}; \draw (10.5,2) node {$3$};
\draw (1.5,4.3) node {$4$}; \draw (4.5,4.3) node {$1$};
\draw (6.5,4.3) node {$3$}; \draw (8.5,4.3) node {$3$};
\draw (10.5,4.3) node {$1$};
\end{tikzpicture}
\end{center}
\caption{Situations from the proofs of Theorems \ref{complbip} and \ref{C5connected}.} \label{second}
\end{figure}
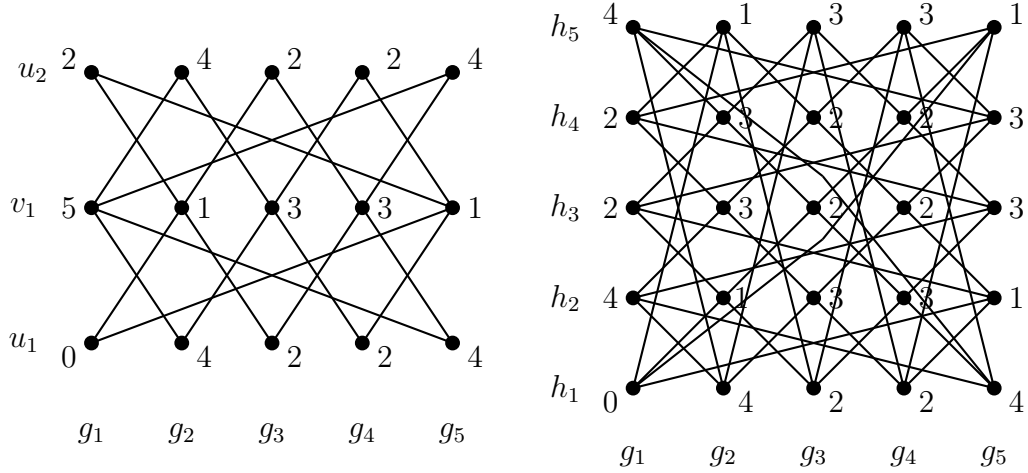

\begin{theorem}\label{C5connected}
For any nonbipartite triangle free $C_5$-connected graphs $G$ and $H$ of diameter two, $$dim_s(G\times H)=\beta(G\Box H).$$
\end{theorem}

\begin{proof}
In order to use Theorem \ref{lem_oellerman} we first describe $(G\times H)_{SR}$. Let $V(G)=\{g_1,\ldots ,g_n\}$
and $V(H)=\{h_1,\ldots ,h_k\}$. Graphs $G$ and $H$ are $C_5$-connected graphs, which imply that their even and odd distances between
arbitrary vertices always exist. Moreover, the even distances are between 0 and 4, while the odd distances are between 1 and 5. Hence, by the distance
formula (\ref{distance}), we can have the distances between 0 and 4 in $G\times H$. We may assume that $g_1g_2g_3g_4g_5g_1$
and $h_1h_2h_3h_4h_5h_1$ are induced five-cycles of triangle free $C_5$-connected graphs $G$ and $H$, respectively. Again, by this
distance formula, it is easy to see that $d_{G\times H}((g_1,h_1),(g_1,h_j))=4$ for $j\in \{2,5\}$ and that $d_{G\times H}((g_1,h_1),(g_j,h_1))=4$ for
$j\in \{2,5\}$. We show that these are the only neighbors of $(g_1,u_1)$ in $(G\times H)_{SR}$. Clearly, these pairs are mutually
maximally distant since they are diametrical vertices.

Next we show that no other vertex of $G\times K_{k,\ell}$ is mutually maximally distant with $(g_1,u_1)$. By the symmetry of a five-cycle
and the commutativity of the direct product we need to present the arguments only for $g_1,g_2$ and $g_3$ and for $h_1,h_2$ and $h_3$.
They are as follows:
\begin{itemize}
\item $(g_1,h_3)\sim (g_2,h_4)$ and $(g_1,h_3)$ is closer to $(g_1,h_1)$ than $(g_2,h_4)$;
\item $(g_2,h_2)\sim (g_3,h_1)$ and $(g_2,h_2)$ is closer to $(g_1,h_1)$ than $(g_3,h_1)$;
\item $(g_2,h_3)\sim (g_1,h_2)$ and $(g_2,h_3)$ is closer to $(g_1,h_1)$ than $(g_1,h_2)$;
\item $(g_3,h_1)\sim (g_4,h_2)$ and $(g_3,h_1)$ is closer to $(g_1,h_1)$ than $(g_4,h_2)$;
\item $(g_3,h_2)\sim (g_2,h_1)$ and $(g_3,h_2)$ is closer to $(g_1,h_1)$ than $(g_2,h_1)$;
\item $(g_3,h_3)\sim (g_2,h_4)$ and $(g_3,h_3)$ is closer to $(g_1,h_1)$ than $(g_2,h_4)$.
\end{itemize}
See the graph $C_5\times C_5$ on the right part of Figure \ref{second}, where the distances from $(g_1,h_1)$ are marked.
Thus, the vertex $(g_1,h_1)$ is adjacent to $(g_1,h_2),(g_1,h_5),(g_2,h_1)$ and $(g_5,h_1)$ in $(G\times K_{k,\ell})_{SR}$.
Continuing with the same arguments, we obtain that $(g_1,u_1)$ is adjacent to all vertices of
$(\{g_1\}\times N_H(h_1))\cup (N_G(g_1)\times \{h_1\})$ in $(G\times H)_{SR}$.
We can use the same arguments for any vertex of $G\times H$ and therefore, we obtain $(G\times H)_{SR}\cong G\Box H$.
By Theorem \ref{lem_oellerman} we have that $dim_s(G\times H)=\beta ((G\times H)_{SR})=\beta(G\Box H)$,
and the proof is completed.
\end{proof}


\end{document}